\documentclass[11pt,a4paper]{amsart}
\usepackage{amsmath}
\usepackage{amssymb}
\usepackage{amsthm}
\usepackage{amsfonts}
\usepackage{mathrsfs}

\usepackage[left=3cm,top=3cm,right=3cm]{geometry}

\usepackage[T1]{fontenc}
\usepackage[latin1]{inputenc}
\usepackage[english]{babel}
\usepackage{comment}

\newtheorem{thm}{Theorem}[section]
\newtheorem{lemma}[thm]{Lemma}

\newtheorem{coro}[thm]{Corollary}
\newtheorem{proposition}[thm]{Proposition}

\theoremstyle{remark}           

\newcommand{\diam}{\operatorname{diam}}

\newcommand{\dist}{\operatorname{dist}}

\newcommand{\loc}{\operatorname{loc}}
\newcommand{\rad}{\operatorname{rad}}

\newcommand{\R}{\mathbb R}
\newcommand{\Om}{\Omega}
\newcommand{\N}{\mathbb N}
\newcommand{\Z}{\mathbb Z}

\newcommand{\W}{\mathcal W}

\newcommand{\Ha}{\mathcal H}

\newcommand{\sub}{\subset}

\newcommand{\ol}{\overline}

\newcommand{\Char}[1]{\chi_{\lower 1.5pt\hbox{$\scriptscriptstyle #1$}}}
\newcommand{\dom}{{d_{\Omega}}}

\newcommand{\vint}[1]{{\mathchoice
          {\mathop{\vrule width 6pt height 3 pt depth -2.5pt
                  \kern -8pt \intop}\nolimits}%
          {\mathop{\vrule width 5pt height 3 pt depth -2.6pt
                  \kern -6pt \intop}\nolimits}%
          {\mathop{\vrule width 5pt height 3 pt depth -2.6pt
                  \kern -6pt \intop}\nolimits}%
          {\mathop{\vrule width 5pt height 3 pt depth -2.6pt
                  \kern -6pt \intop}\nolimits}}_{\!\!\!\! #1}}

\newcommand{\capp}[1]{\operatorname{cap}_{#1}}

\DeclareMathOperator{\dima}{dim_A}

\DeclareMathOperator{\dimh}{dim_H}

\DeclareMathOperator{\codima}{co-dim_A}

\DeclareMathOperator{\codimh}{co-dim_H}
\DeclareMathOperator*{\essinf}{ess\,inf}

\begin{document}

\title
{Quasiadditivity of variational capacity}
\author{Juha Lehrbäck and Nageswari Shanmugalingam} 
\date{\today}

\address{J.L.:~Department of Mathematics and Statistics, 
P.O. Box 35 (MaD), 
FIN-40014 University of Jyv\"askyl\"a, 
Finland}
\email{\tt juha.lehrback@jyu.fi}

\address{N.S.:~Department of Mathematical Sciences,
P.O. Box 210025,
University of Cincinnati,
Cincinnati, OH 45221-0025, U.S.A.}
\email{\tt shanmun@uc.edu}

\thanks{The first author acknowledges the support of
the Academy of Finland,  grant no.\ 252108. The second author was partially supported by
a grant from the Simons Foundation, Collaboration Grant \#~200474 and NSF grant DMS-1200915.}


\begin{abstract}
We study the quasiadditivity property
(a version of superadditivity with a multiplicative constant)
of variational capacity in metric spaces
with respect to Whitney type covers.
We characterize this property in terms of
a Mazya type capacity condition, and also
explore the close relation between quasiadditivity
and Hardy's inequality. 
\end{abstract}

\maketitle

\section{Introduction}

Given an open set $\Om\subset\R^n$ and a subset $E\subset\Om$, the relative $p$-capacity
$\capp{p}(E,\Om)$ measures the minimal energy needed by a Sobolev function that vanishes on $\partial\Om$
to take on a value at least $1$ on $E$. In potential theory this quantity can also be seen as a measure of 
the amount of rectifiable curves connecting $E$ to $\partial\Om$.  Hence, greater the amount of $\partial\Om$
that is close to $E$ the larger the relative $p$-capacity is. 

It can be seen that $E\mapsto\capp{p}(E,\Om)$ is an outer measure on subsets of $\Om$.
In particular, capacity is countably sub-additivite: if
$E_k\sub\Omega$, $k\in I\subset\N$, then 
\begin{equation}\label{eq:subadd}
    \capp{p}\Big(\bigcup_{k\in I}E_k,\Om\Big)\le \sum_{k\in I}\capp{p}(E_k,\Om).
\end{equation}
Unlike for Borel regular measures, the equality in \eqref{eq:subadd} does not (usually) hold even for nice, well-separated sets. Indeed, the only sets that are measurable with respect to $\capp{p}$-outer measure are
sets of zero capacity and their complements, see for example~\cite[Theorem~4.8]{Sio}
or~\cite[Theorem~2]{Don}. 
Nevertheless, in some cases a converse to \eqref{eq:subadd}, with a multiplicative constant, can be shown to hold for certain of unions of sets; this is called the \emph{quasiadditivity property} of capacity. More precisely, we say that the $p$-capacity relative to an open set $\Om$ is \emph{quasiadditive} with respect to a given cover (or a decomposition) $\W$ of $\Om$ if
there is a constant $A>0$ such that
\[
  \sum_{Q\in\W}\capp{p}(E\cap Q,\Om)\le A\, \capp{p}(E,\Om)
\]
for all $E\subset\Om$.

The quasiadditivity property (for the linear case $p=2$)
was first considered by Landkof \cite[Lemma 5.5]{land} (without the name) and Adams \cite{Ad} for Riesz (and Bessel) capacities with respect to annular decompositions of $\R^n\setminus\{0\}$. Aikawa generalized these results in~\cite{A}, where he showed that if the complement $\R^n\setminus\Om$ has a sufficiently small dimension (formulated in terms of a local version of packing condition), then the Riesz capacity of $\R^n$ is quasiadditive with respect to Whitney decompositions of $\Om$. On the other hand, in \cite{A2} Aikawa considered the Green capacity (obtained via the Green energy) and demonstrated that if $\R^n\setminus\Om$ is uniformly regular (or, equivalently, uniformly $2$-fat),
then the Green capacity is quasiadditive with respect to Whitney decompositions of $\Omega$. Note that in this case, conversely to the result of \cite{A}, the complement $\R^n\setminus\Om$ has a large dimension. 
A good survey of these topics in the Euclidean setting can be found in~\cite[Section 7 of Part II]{AE}. See also
\cite[Section 16 of Part I]{AE} and \cite{AB} for related results for nonlinear ($p\ne 2$) setting for which
decompositions other than the Whitney decomposition are used.

The aim of this note is to study the quasiadditivity problem for the relative $p$-capacity with respect to Whitney type covers in the setting of complete metric measure spaces satisfying the `standard' structural assumptions (see Section \ref{sect:metric}). Nevertheless, most of our results are new for $p\ne 2$ 
(and for $p=2$, obtained via new methods) even in Euclidean spaces. Part of our motivation stems from a need to clarify the relation between quasiadditivity and Hardy's inequality (a Sobolev-type inequality weighted with a power of the distance-to-boundary function; see Section \ref{sect:hardy}). Glimpses of a connection between these concepts (and the related dimension bound of Aikawa from \cite{A}) have appeared e.g.\ in \cite{A2,AE,KZ,LMM,LT,W2}, but now our main result --- Theorem~\ref{thm:all}, a characterization of quasiadditivity in terms of a Mazya type capacity estimate ---
reveals a simple equivalence between quasiadditivity and Hardy's inequalities (Corollary~\ref{coro:q-add from hardy}).
These results also link quasiadditivity and the geometry of the boundary (or the complement) of the open set $\Omega$. 

The organization of this note is as follows. In Section \ref{sect:prelim} we recall some of the necessary background material: The basic assumptions, the notions of (co)dimension for metric spaces, Sobolev type spaces and the related capacities, and Whitney covers (substitutes of the classical Whitney decompositions for our more general spaces). Since our proofs are largely based on potential-theoretic (rather than PDE-based) tools, an overview of these is given at the end of  Section~\ref{sect:prelim}; of a particular importance for us is the weak Harnack inequality for superminimizers. 
Section \ref{sect:qadd} contains our main characterization of quasiadditivity and the aforementioned connection with Hardy's inequalities. A concrete outcome of these considerations is that the uniform $p$-fatness of $X\setminus\Omega$ guarantees the quasiadditivity for the relative $p$-capacity in $\Omega$ for $1<p<\infty$.

Aikawa's dimension bound $\dima(\R^n\setminus \Omega)<n-p$ from \cite{A} translates to more general metric spaces as $\codima(X\setminus \Omega)>p$. We show in Section \ref{sect:q and a} that this bound, together with an additional discrete John type condition, is sufficient for the relative $p$-capacity to be quasiadditive with respect to Whitney covers of $\Omega$. 
Finally, in Section \ref{sect:combo} we explain how the results involving a large complement (uniform $p$-fatness) or a small complement (Aikawa's condition) can be combined, allowing us to deal with more general open sets whose complements consist of parts of different sizes.

For the notation we remark that $C$ 
will denote positive constants whose value is
not necessarily the same at each occurrence. 
If there exist constants $c_1,c_2>0$ such that $c_1\,F\le
G\le c_2F$, we sometimes write $F\approx G$ and say that $F$ and $G$ are comparable.

\section{Preliminaries}\label{sect:prelim}

\subsection{Doubling metric spaces}\label{sect:metric}

We assume 
throughout this note that $X=(X,d,\mu)$ is a complete metric measure space, where
$\mu$ is a Borel measure supported on $X$, with $0<\mu(B)<\infty$ whenever $B=B(x,r)$ is an open ball in $X$, and 
that $\mu$ is \emph{doubling}, that is, there is a constant $C>0$ such that whenever $x\in X$ and $r>0$, we have
\[
  \mu(B(x,2r))\le C\, \mu(B(x,r)).
\]
We make the tacit assumption that each ball $B\sub X$ has a fixed center $x_B$ and radius $\rad(B)$, and thus notation such as $\lambda B = B(x_B,\lambda \rad(B))$ is well-defined for all $\lambda>0$. 

If $\mu$ is a doubling measure, then by iterating the doubling condition we find constants $Q>0$ and $C>0$ such that 
\begin{equation*}\label{eq: d dim*}
\frac{\mu(B(y,r))}{\mu(B(x,R))}\ge C\Bigl(\frac rR\Bigr)^Q
\end{equation*}
whenever $0<r\le R<\diam X$ and $y\in B(x,R)$.  Furthermore, if $X$ is connected
(this is guaranteed in our setting by the below-mentioned Poincar\'e inequalities), then there exists a constant $Q_u>0$ such that 
for all $0<r<R<\diam X$ and $y\in B(x,R)$,
\begin{equation}\label{eq: d dim conv}
   \frac{\mu(B(y,r))}{\mu(B(x,R))}\le C\Bigl(\frac rR\Bigr)^{Q_u}.
\end{equation}
In general, $1\le Q_u\le Q$.
However, if we have uniform upper and lower bounds for the measures of
the balls, i.e.\ 
\[
c^{-1} r^Q\leq \mu(B(x,r))\leq c r^Q
\]
for every $x\in X$ and all $0<r<\diam(X)$,
we say that the measure $\mu$ is \emph{(Ahlfors) $Q$-regular}.

When working with a (non-regular) doubling measure $\mu$, it is often convenient to describe the sizes of sets in terms of \emph{codimensions} (instead of dimensions). For instance, the \emph{Hausdorff codimension} of $E\sub X$ (with respect to $\mu$) is the number
\[\codimh(E) = \sup\big\{q\geq 0 : \Ha_\infty^{\mu,q}(E)=0\big\},\]
where 
\[
\Ha_\infty^{\mu,q}(E) = \inf\bigg\{ \sum_{k} \rad(B_k)^{-q}\mu(B_k) : E \subset \bigcup_{k} B_k\bigg\}
\] 
is the \emph{$q$-codimensional Hausdorff content}. If $\mu$ is $Q$-regular, then we have for all $E\sub X$ that $Q\,-\,\codimh(E)=\dimh(E)$, the usual Hausdorff dimension.

Another notion of codimension that will be useful for us is the \emph{Aikawa codimension}: For $E\sub X$ we define
$\codima(E)$ as the supremum of all $q\geq 0$ for which there exists a constant $C_{q}$ such that 
\begin{equation*}\label{AikDim*} 
\int_{B(x,r)}\text{dist}(y,E)^{-q}\,d\mu(y)\le C_{q}r^{-q}\mu(B(x,r))  
\end{equation*} for every $x\in E$ and all $0<r<\diam(E)$. 
Here we interpret the integral to be $+\infty$ if $q>0$ and $E$ has positive measure. It is not hard to see that 
$\codima(E)\leq\codimh(E)$ for all $E\sub X$ (cf.\ \cite{LT}).
If $\mu$ is Ahlfors $Q$-regular, then we could define the \emph{Aikawa dimension} of a set $E\subset X$ to be the number $\dima(E) = Q - \codima(E)$. Nevertheless, it was shown in~\cite{LT} that for subsets
of Ahlfors regular metric measure spaces this concept is actually equal to 
the \emph{Assouad dimension} of the subset; see \cite{luukas} for the basic properties of the Assouad dimension.

\subsection{Sobolev-type function spaces in the metric setting}\label{sect:sobo}

There are many analogs of Sobolev-type function spaces in the metric setting. The one considered in this note
is based on the notion of upper gradients, generalizing the fundamental theorem of calculus. Given
a measurable function $f:X\to[-\infty,\infty]$, we say that a Borel measurable non-negative function $g$ on $X$
is an \emph{upper gradient} of $f$ if whenever $\gamma$ is a compact rectifiable curve in $X$, we have
\[
  |f(y)-f(x)|\le \, \int_\gamma g\, ds.
\]
Here $x,y$ denote the two endpoints of $\gamma$, and the above condition should be interpreted as claiming that
$\int_\gamma g\, ds=\infty$ whenever at least one of $|f(x)|, |f(y)|$ is infinite. 
See~\cite{HEI} and~\cite{BB} for a good discussion on the notion of upper gradients. 
Using upper gradients as a substitute
for modulus of the weak derivative, we define the norm
\[
  \Vert f\Vert_{N^{1,p}(X)}:= \left(\int_X|f|^p\, d\mu\, +\, \inf_g\int_X g^p\, d\mu\right)^{1/p},
\]
where the infimum is taken over all upper gradients $g$ of $f$. The \emph{Newtonian space} $N^{1,p}(X)$ is
the space 
\[
  \{f:X\to[-\infty,\infty]\, :\, \Vert f\Vert_{N^{1,p}(X)}<\infty\}/\sim,
\] 
where the equivalence $\sim$ is given by $u\sim v$ if and only if $\Vert u-v\Vert_{N^{1,p}(X)}=0$
(see~\cite{Sh} or~\cite{BB} for more on this function space).

In addition to the doubling property,
we will also assume throughout that the space $X$ supports a $(1,p)$-Poincar\'e inequality, that is, there exist constants 
$C>0$ and $\lambda\ge 1$ such that whenever $B=B(x,r)\subset X$ and $g$ is an upper gradient of a measurable function $f$, we have
\[
    \vint{B} |f-f_B|\, d\mu \le C\, r\, \left(\,\vint{\lambda B}g^p\, d\mu\right)^{1/p}
\]
where 
\[
   f_B:=\frac{1}{\mu(B)}\, \int_B f\, d\mu =:\vint{B}\, f\, d\mu.
\]

Different notions of capacity are of fundamental importance in many questions related to the behavior of the functions belonging to a certain class. Given a set $E\subset X$, the \emph{total $p$-capacity} of $E$, denoted $\text{Cap}_p(E)$, is the infimum of $\Vert u\Vert_{N^{1,p}(X)}^p$ over all functions $u$ such that $u\ge 1$ on $E$. Just as sets of measure zero play the role of indeterminacy in the theory of Lebesgue spaces $L^p(X)$, sets of total $p$-capacity zero play the corresponding role in the theory of Sobolev type spaces; see~\cite{BB} or~\cite{Sh} for details. We say that a property holds ($p$-)\emph{quasieverywhere} ($p$-q.e.) if the exeptional set is of zero total capacity.

When the examinations are taking place in an open set $\Omega\sub X$, then a more appropriate version of capacity is \emph{the relative $p$-capacity}. For a measurable set $E\subset \Om$ this is defined as the number
\[
    \capp{p}(E,\Om):= \inf_u\inf_{g_u} \int_X g_u^p\, d\mu,
\]
where the infimum is taken over all $u\in N^{1,p}(X)$ with $u=0$ on $X\setminus\Om$, $u=1$ on $E$, 
$0\le u\le 1$, and over all upper gradients $g_u$ of $u$. A function $u$ satisfying the above conditions is called a \emph{capacity test function} for $E$. Should no such function $u$ exist, we set
$\capp{p}(E,\Om)=\infty$.

When the variational capacity is taken with respect to $\Omega=X$, it may be the case that $\capp{p}(E,X)=0$ for all bounded $E\sub X$; this is certainly true if $X$ is bounded. If $X$ is unbounded and still $\capp{p}(E,X)=0$ for all bounded $E\sub X$, then $X$ is called \emph{$p$-parabolic}, but if $\capp{p}(E,X)>0$ for some bounded $E\sub X$, then $X$ is \emph{$p$-hyperbolic}. These notions will be relevant in the considerations of Section \ref{sect:q and a}.
See~\cite{Holo} and~\cite{HS} for more on parabolic and hyperbolic spaces. Notice in particular that if $X$ is bounded or $p$-parabolic and $\Omega\sub X$ is such that $\text{Cap}_p(X\setminus\Omega)=0$, then $\capp{p}(E,\Omega)=0$ for every $E\sub\Omega$.

Besides measuring small (exeptional) sets, the relative capacity can also be used to give conditions for the largeness of sets. For instance, a closed set $E\sub X$ is said to be \emph{uniformly $p$-fat} if there exists a constant $C>0$ such that
\[
 \capp{p}(E\cap B, 2B)\geq  C \capp{p}(B, 2B)
\]
for all balls $B$ centered at $E$. Here $\capp{p}(B,2B)$ is actually comparable with $\rad(B)^{-p}\mu(B)$ for all balls $B$ of radius less than $\diam(X)/8$. See \cite[Chapter~6]{BB} for this and other basic properties of the total and the variational capacity on metric spaces. We remark that the uniform $p$-fatness can also be characterized using uniform density conditions for Hausdorff contents; see e.g.\ \cite{KLT}.

Recall that we say the variational $p$-capacity $\capp{p}(\cdot,\Omega)$ to be \emph{quasiadditive}
with respect to a decomposition or a cover $\W$ of $\Omega$ if there exists a constant $A>0$ such that
\begin{equation*}\label{eq:quasiadd*}
    \sum_{Q\in\W} \capp{p}(E\cap Q,\Omega) \leq A\capp{p}(E,\Omega)
 \end{equation*}
for every $E\subset\Omega$. In the next subsection we discuss the one particular 
family of covers we are concerned with.

\subsection{Whitney covers}\label{sect:W}

Let $\Omega\sub X$ be an open set. We often write $\dom(x)=\dist(x,X\setminus\Omega)$ for $x\in\Omega$. 
When $0<c\leq 1/2$, we fix a Whitney type cover $\W_c(\Omega)=\{B_i\}_{i\in \N}$, $\Omega\subset\bigcup_{i\in\N}B_i$, consisting of balls
$B_i=B(x_i,c\dom(x_i))$, $x_i\in\Omega$, such that the balls $B_i$ have uniformly controlled overlap: there exists $1\leq C<\infty$ such that $\sum_i \Char{B_i}\leq C$. Such a cover can always be constructed by considering maximal packings (or, alternatively, `$5r$'-covers) of the sets $\{x\in\Omega : 2^{-k-1}\leq \dom(x) < 2^{-k}\}$, $k\in\Z$, with the balls of the above type.
In pathological situations we allow $B_i=\emptyset$ for some $i$, if necessary.

In our proofs, we need to be able to dilate the Whitney balls without having too much
overlap; the existence of such a cover is established in the next 
(elementary) lemma. For a proof, see e.g.\ \cite{BBS} (Theorem~3.1 together with
Lemma~3.4) or~\cite[Chapter~3]{HKST}.

\begin{lemma}\label{lemma:no overlap}
Let $\Omega\sub X$ be an open set. Fix $L\geq 1$ and let
$\W_c(\Omega)=\{B_i\}_{i\in \N}$ be a Whitney cover of $\Omega$
with $c\leq (3L)^{-1}$. Then the balls $LB_i$ have
a uniformly bounded overlap. 
\end{lemma}

In the proof of our main characterization of the quasiadditivity,
we will need for Whitney balls $B\in\W_c(\Omega)$ 
the estimate
\begin{equation}\label{eq:est for cap}
  c_1 \rad(B)^{-p} \mu(B) \leq \capp{p}(B,\Omega) \leq c_2 \rad(B)^{-p}\mu(B),
\end{equation}
where $c_1,c_2$ may depend on $\W_c(\Omega)$ but not on the particular $B\in\W_c(\Omega)$.
The upper bound in~\eqref{eq:est for cap} is always true in our setting, and can be proved 
almost immediately by using only the doubling condition and the test function
$u(x)=[1-d_\Om(x)/\rad(B)]_+$.
The lower bound is a bit more involved, and can in fact fail in
some spaces satisfying our basic assumptions. Thus, in the cases where we need
the lower bound, we will need to have some extra assumptions on $\Omega$ or $X$, 
e.g.\ those given in Lemma~\ref{lemma:cap low bound} below.
However, we stress that the lower bound in~\eqref{eq:est for cap} is only
needed in the proof of Lemma~\ref{lemma:main for balls}, which on the
other hand is only used to prove the Main Theorem~\ref{thm:all}, 
and then once again in the considerations of Section~\ref{sect:combo},
and thus these are the only instances where such assumptions are needed.

Another important case when the lower bound is valid is when
$X\setminus\Omega$ is uniformly $p$-fat; the bound then follows easily
(e.g.) from the $p$-Hardy inequality (see Section \ref{sect:hardy})
for capacity test functions of $B\in\W_c(\Omega)$.
In this case we only need the standing assumptions that $\mu$ is doubling 
and $X$ supports a $(1,p)$-Poincar\'e inequality. 

If we want to weaken the assumption on $X\setminus\Omega$, we need to
assume more on $\mu$ and~$p$. In the following lemma we chose the condition
that $X$ is unbounded (in which case we have $\mu(X)=\infty$ by~\eqref{eq: d dim conv}). 
However, if $X$ happens to be bounded, then we could impose a further 
condition on $\Om$ instead, such as $\diam(\Omega)<2\diam(X)$,
in which case the constants depend on $\gamma=\mu(\Om)/\mu(X)<1$. 
The proof in this case is similar to the one below. 
Recall here that $Q_u$ is the exponent from the `upper mass bound' \eqref{eq: d dim conv}.

\begin{lemma}\label{lemma:cap low bound}
Assume that $1\leq p < Q_u$ and $X$ is unbounded.
Then 
\begin{equation*}\label{eq:low est for cap}
  c\, \rad(B)^{-p}\mu(B) \leq \capp{p}(B,\Omega) 
\end{equation*}
for all Whitney balls $B\in\W_c(\Omega)$.
\end{lemma}

\begin{proof}
Let $u\in N^{1,p}(\Omega)$ be a capacity test function for $B=B(x,r)\in\W_c(\Omega)$,
and let $g\in L^p(\Omega)$ be an upper gradient of $u$. 
For positive integers $j$ let $B_j=2^j B$ and $r_j=2^j r$. 
As $u\in L^p(X)$ and $\mu(X)=\infty$ by the 
unboundedness of $X$ (here we use~\eqref{eq: d dim conv}),
there exists $K\in\N$ (depending on $u$) such that $u_{B_K}<1/2$. 
On the other hand, because $u\ge 1$ on $B$ we know that $u_B=1$.
Now a standard telescoping argument using the $(1,p)$-Poincar\'e inequality
yields
\begin{equation}\label{eq:chainagain}
\tfrac 1 2 \leq |u_B-u_{B_K}|\leq 
C \sum_{j=1}^K r_j \bigg(\,\,\vint{\lambda B_j} g^p \,d\mu\bigg)^{1/p}.
\end{equation}
It follows that there exists a constant $C_0>0$ and some $1\leq j_0 \leq K$ 
such that 
\[
 (r/r_{j_0})^{(Q_u-p)/p} = 
      (2^{-j_0})^{(Q_u-p)/p}\leq C_0 r_{j_0} \mu(\lambda B_{j_0})^{-1/p} 
        \bigg(\int_{\lambda B_{j_0}} g^p \,d\mu\bigg)^{1/p}
\]
(otherwise \eqref{eq:chainagain} would lead to a contradicition when
compared to a geometric series).
Thus we obtain,
using also \eqref{eq: d dim conv} and the assumption $1<p<Q_u$, that
\[
\int_{\Omega} g^p\,d\mu \geq C (r/r_{j_0})^{Q_u-p} r_{j_0}^{-p} \mu(\lambda B_{j_0})  
  \geq C r^{-p} \mu(B),
\]
as desired. 
\end{proof}

Let us record the following easy 
consequence of estimate \eqref{eq:est for cap}
for unions of Whitney balls. 

\begin{lemma}\label{lemma:main for balls}
Let $\Omega\sub X$ be an open set and let
$\W_c(\Omega)=\{B_i\}_{i\in \N}$ be a Whitney cover of $\Omega$
for which the lower bound in~\eqref{eq:est for cap} holds.
Furthermore, let
$U\sub \Omega$ be a union of Whitney balls, 
i.e., $U=\bigcup_{i\in I} B_i$ for some $I\sub \N$.
Then
\[\int_U \dom(x)^{-p}\,d\mu \approx
\sum_{i\in I} \capp{p}(B_i,\Omega).
\]
\end{lemma}

\begin{proof}
 This follows from the fact that $\dom(x)\approx \rad(B_i)$ for all $x\in B_i$, with
 constants only depending on $c$,
 and that $\capp{p}(B_i,\Omega)\approx \rad(B_i)^{-p}\mu(B_i)$ by \eqref{eq:est for cap}.
\end{proof}

\subsection{Existence and properties of $p$-potentials}\label{sect:nonlin-pot-thry}

In computing the relative $p$-capacity $\capp{p}(E,\Omega)$, should this capacity be finite, 
we can find a minimzing sequence of capacity test functions $u_k\in N^{1,p}(X)$, i.e.,
\[
   \lim_{k\to\infty} \inf_{g_{u_k}}\int_X g_{u_k}^p\, d\mu=\capp{p}(E,\Omega).
\]
We will assume throughout that $1<p<\infty$; hence $L^p(X)$ 
is reflexive,
and so a standard variational argument using Mazur's lemma on $L^p(X)$ (as in 
Lemma~\ref{lem-existence} below)
tells us that if $\Omega\sub X$ is bounded and $\text{Cap}_p(X\setminus \Omega)>0$, then
there is a \emph{$p$-potential} $u\in N^{1,p}(X)$ such that $0\le u\le 1$ on $X$, $u=1$ on $E$, $u=0$ on $X\setminus\Omega$, and 
\[
   \inf_{g_u}\int_\Omega g_u^p\, d\mu=\inf_{g_u}\int_X g_u^p\, d\mu=\capp{p}(E,\Omega).
\]
Such a $p$-potential is \emph{unique} because $X$ supports a $(1,p)$-Poincar\'e inequality;
see~\cite{Sh2} for more details. 
Nevertheless, the following more general lemma tells us that
a $p$-potential $u\in N^{1,p}_{\loc}(X)$ exists in more general cases (e.g.\ if $\text{Cap}_p(X\setminus\Om)=0$) as well; though, if $X$ is 
$p$-parabolic, we would have $u$ be a constant. In addition,
the below proof shows that the reflexivity of $N^{1,p}(X)$ is actually not needed for the
existence of $p$-potentials.

\begin{lemma} \label{lem-existence}
Let $1<p<\infty$ 
and let $E\subset\Omega$ be such that $\capp{p}(E,\Omega)<\infty$. Then there is a function $u\in N^{1,p}_{\loc}(X)$
such that $u=1$ $p$-q.e.~on $E$, $u=0$ $p$-q.e.~in $X\setminus\Omega$, $0\le u\le 1$ on $X$, and 
\[
   \inf_{g_u}\int_{\Omega} g_u^p\, d\mu=\capp{p}(E,\Omega).
\]
\end{lemma}

\begin{proof}
If $\Omega$ is bounded, then the following proof can be easily modified, or the
results of~\cite{Sh2} can be used to obtain the desired conclusion. Hence here we will only
give the proof for the case that $\Omega$, and hence $X$, is unbounded.

For each $u\in N^{1,p}(X)$ there is a minimal ($p$-weak) upper gradient $g_u\in L^p(X)$; see 
for example~\cite{BB}. Hence, from now on, we let $g_u$ denote this minimal upper gradient of $u$.

Let $\{u_k\}_{k\in\N}$ be a sequence of functions in $N^{1,p}(X)$ that satisfy $0\le u_k\le 1$ on $X$,
$u_k=0$ on $X\setminus\Omega$ $p$-q.e., $u_k=1$ $p$-q.e.~on $E$, and 
\[
   \lim_{k\to\infty} \int_\Omega g_{u_k}^p\, d\mu\, =\, \capp{p}(E,\Omega).
\]

Fix $x_0\in\Omega$ and for each positive integer $n$ let $B_n=B(x_0,n)$. Given that 
$0\le u_k\le 1$, the sequences $\{u_k\}_k$ and $\{g_{u_k}\}_k$ are bounded in
$L^p(B_n)$ for each positive integer $n$, and hence because $1<p<\infty$, the uniform 
convexity of $L^p(B_n)$ together with a Cantor diagonalization argument tells us that 
$\{u_k\}_k$ converges weakly to a function $\hat{u}$ in $L^p_{\loc}(X)$ and that 
$\{g_{u_k}\}_k$ converges weakly to $g$ in $L^p_{\loc}(X)$. 

Finally, an application of~\cite[Lemma~3.1]{KaS} to $\{u_k\}_k$ and $\{g_{u_k}\}_k$ in $B_n$
allows us to modify $\hat{u}$ on a set of measure zero to obtain a function $u\in L^p_{\loc}(X)$
that has $g$ as a $p$-weak upper gradient, and furthermore, from~\cite[Lemma~3.1]{KaS} 
and~\cite[Proof of Theorem~3.7]{Sh}, we can conclude that $u=1$ $p$-q.e.~on $E$, 
$u=0$ $p$-q.e.~on $X\setminus\Omega$, and that 
\[
  \int_\Omega g_u^p\, d\mu\le \int_\Omega g^p\, d\mu\
     \le \lim_{k\to\infty} \int_\Omega g_{u_k}^p\, d\mu\, =\, \capp{p}(E,\Omega).
\]
This, together with the definition of $\capp{p}(E,\Omega)$ now completes the proof.
\end{proof}

The results from~\cite{KS} show that such $u$, if non-constant, satisfies $u>0$ on $\Omega$ with
$u<1$ on $X\setminus\overline{E}$.
Of course, should $\capp{p}(E,\Omega)$ be infinite, then no such $u$ exists. 

It is clear that the $p$-potential $u$, corresponding to $E\sub X$, 
has the property that $u$ is a \emph{$p$-superminimizer} in $\Omega$ and a
\emph{$p$-subminimizer} in $X\setminus\overline{E}$; in particular,
$u$ is a \emph{$p$-minimizer} in $\Omega\setminus E$. Here, we say that a function $v\in N^{1,p}_{\loc}(X)$ is
a $p$-superminimizer in an open set $U\subset X$ if, whenever $\varphi\in N^{1,p}(X)$ is a non-negative function
such that $\varphi=0$ on $X\setminus U$, we have
\[
   \inf_{g_v}\int_{\text{supt}(\varphi)} g_v^p\, d\mu
      \le \inf_{g_{v+\varphi}} \int_{\text{supt}(\varphi)} g_{v+\varphi}^p\, d\mu,
\]
and $v$ is a $p$-subminimizer in $U$ if $-v$ is a $p$-superminimizer in $U$. We refer the interested reader to~\cite{KiMa}
for information on minimizers; see also \cite{BB}. In particular, it is known that 
if $v$ is a $p$-superminimizer in $U$ and $w\in N^{1,p}_{\loc}(X)$ is a $p$-minimizer in $U$ such that $w\le v$ holds
$p$-q.e.\ on $X\setminus U$, 
then $w\le v$ on $U$ as well. This is the so-called \emph{comparison principle}.
Notice also that if $\text{Cap}_p(X\setminus\Omega)=0$ and 
$v\in N^{1,p}_{\loc}(X)$ is a minimizer in $\Omega$, then $v$ is a minimizer in $X$; moreover, 
in this case, if $u$ is
a $p$-potential for $\capp{p}(E,\Omega)$ then it is a $p$-potential for $\capp{p}(E,X)$.

In the proofs of our results the following weak Harnack inequality for $p$-superminimizers is of fundamental importance. See~\cite{KS} for a proof of this lemma.

\begin{lemma}[Weak Harnack]\label{lemma:weak harnack}
There exists constants $A>0$, $C_A\geq 1$, and $q>0$ such that
if $u$ is a $p$-superminimizer in $C_A B\subset\Omega$, then
\[
\bigg(\,\,\vint{2B} u^q\,d\mu \bigg)^{1/q} \leq A \essinf_B u.
\]
\end{lemma}

\section{Characterizations of quasiadditivity}\label{sect:qadd}

In this section we prove the main result of this note, Theorem~\ref{thm:all}, and provide a 
connection between quasiadditivity and $p$-Hardy inequalities. Recall that we always assume that $1<p<\infty$.

\subsection{The main characterization}

We begin by showing that quasiadditivity property for unions of balls is in fact
sufficient for the quasiadditivity for general sets. Below $C_A$ is the constant from the
weak Harnack inequality and $\lambda$ is the dilatation constant from the $p$-Poincar\'e inequality.

\begin{proposition}\label{prop:q-add}
Let $\Omega\sub X$ be an open 
set with $\Om\ne X$ and let
$\W_c(\Omega)=\{B_i\}_{i\in I}$ be a Whitney cover of  $\Omega$ 
with $c < \min\{(C_A)^{-1},(30\lambda)^{-1}\}$. 
Assume that the quasiadditivity condition holds
for unions of Whitney balls, i.e., 
if $U=\bigcup_{i\in I} B_i$ for some $I\subset\N$, $B_i\in \W_c(\Omega)$, 
then
\[
\sum_{i\in I} \capp{p}(B_i,\Omega) \leq C_1  \capp{p}(U,\Omega).
\]
Then the capacity $\capp{p}(\cdot,\Omega)$ is 
quasiadditive with respect to $\W_c(\Omega)$,
i.e., there exists a constant $C>0$ such that
\[
\sum_{i\in \N} \capp{p}(E\cap B_i,\Omega) \leq C \capp{p}(E,\Omega)
\]
for every $E\subset\Omega$.
\end{proposition}

\begin{proof}
The structure of the proof is based on the idea of Aikawa \cite{A}, but given the non-linear nature
of our setting, the tools we employ are completely different.
Let $E\subset\Omega$. If  the 
relative capacity $\capp{p}(E,\Om)$ is infinite, then 
the desired inequality would follow. Therefore we assume that $\capp{p}(E,\Om)<\infty$, and let
$u\in N^{1,p}_{\loc}(\Omega)$ be the $p$-potential corresponding to this capacity.
If $\capp{p}(E,\Om)=0$, then by the monotonicity of $\capp{p}(\cdot,\Omega)$, 
each term in the sum on the left-hand side of the desired
inequality is also zero, and the claim follows. Therefore we will assume that $\capp{p}(E,\Om)>0$,
and so the $p$-potential $u$ is non-constant.

Write $E_i=E\cap B_i$.
Choose $C_0=\big(\frac 1 4\big)^{1/q} \frac 1 A$,
where $q$ and $A$ are the constants from the weak Harnack inequality
(Lemma~\ref{lemma:weak harnack}).
We divide the union $E=\bigcup_{i\in\N} E_i$ into the following two parts:
If $u(x)\geq C_0$ for q.e.\ $x\in B_i$, then $i\in I_1$, and otherwise
$i\in I_2$. Note also that the indices $i$ for which 
$\capp{p}(E_i,\Om)=0$ do not contribute to the sum on the left-hand side
of the desired quasiadditivity inequality. Hence in the following argument we only consider
the indices $i$ for which $\capp{p}(E_i,\Om)>0$.

It is immediate that $\frac u {C_0}$ is an admissible test function for
\[
\capp{p}\Big(\bigcup_{i\in I_1}B_i,\Omega\Big).
\]
Thus, using the assumption that quasiadditivity holds for unions of Whitney balls,
we obtain for all upper gradients $g_u$ of $u$ that
\begin{equation}\label{eq:I_1 is ok}\begin{split}
\sum_{i\in I_1}\capp{p}(E_i,\Omega) & \leq \sum_{i\in I_1}\capp{p}(B_i,\Omega)
\leq C_1 \capp{p}\Big(\bigcup_{i\in I_1}B_i,\Omega\Big) 
\leq C_1 \Big(\frac 1 {C_0}\Big)^p \int_\Omega g_u^p\,d\mu. 
\end{split}
\end{equation}

On the other hand, if $i\in I_2$, then by the weak Harnack inequality
\[
\bigg(\,\,\vint{2B_i} u^q\,d\mu \bigg)^{1/q} \leq A \essinf_{B_i} u < C_0 A,
\]
and thus
\begin{equation}\label{eq:u^q small}
 \int_{2B_i} u^q d\mu \leq (C_0 A)^q \mu(2B_i) = \tfrac 1 4 \mu(2B_i).
\end{equation}
Since $0\leq u^q\leq 1$, it follows that
for the set $U_i=\big\{x\in 2B_i : u^q\geq \frac 1 2 \big\}$ we have
$\mu(U_i)\leq \frac 1 2 \mu(2B_i)$. 

Now write $v=1-u$, whence $v\in N^{1,p}_{\loc}(X)$ and 
the class of upper gradients of $u$ is precisely the class of upper gradients of $v$ as well.
Also, $U_i=\big\{x\in 2B_i : v\leq 1-\big(\frac 1 2\big)^{1/q} \big\}$,
and so
\[
\tfrac 1 2 \mu(2B_i)\leq \mu(2B_i\setminus U_i) = \big\{x\in 2B_i : v\geq 1-\big(\tfrac 1 2\big)^{1/q} \big\}.
\]
This gives a positive lower bound $c_1$ for the mean-value of $v^p$ in $2B_i$;
\[
  (v^p)_{2B_i}\ge \tfrac12\left(1-\tfrac{1}{2^{1/q}}\right)^{p}=:c_1.
\]
We can now use the well-known Mazya's version of the (Sobolev--)Poincar\'e inequality
(see e.g.\ \cite[Chapter~10]{mazya}, and \cite[Proposition~3.2]{B} for the
metric space version):
\begin{equation}\label{eq:mazya}
c_1<\vint{2B_i}v^p\,d\mu\le\frac C{\capp{p}(B_i\cap\{v=0\},2B_i)}
\int_{10\lambda B_i}g_v^p\,d\mu,
\end{equation}
where $g_v$ is an arbitrary upper gradient of $v$ (and thus of $u$ as well).
Since $E_i= B_i\cap\{v=0\}$ by the comparison principle, 
it follows from \eqref{eq:mazya} that
\[\begin{split}
\capp{p}(E_i,\Omega) & \leq \capp{p}(E_i,2B_i) =  \capp{p}\big(B_i\cap\{v=0\},2B_i\big) 
\leq C' \int_{10\lambda B_i}g_v^p\,d\mu.
\end{split}\]
Using this and the fact that
the balls $10\lambda B_i$ do not overlap too much,
guaranteed by our choice of the parameter $c$ (with $L\ge 10\lambda$)
and Lemma \ref{lemma:no overlap}, we conclude that
\begin{equation}\label{eq:I_2 is ok}
\sum_{i\in I_2}\capp{p}(E_i,\Omega)\leq C' \sum_{i\in I_2} \int_{10\lambda B_i}g_v^p\,d\mu
\leq C \int_{\Omega}g_v^p\,d\mu. 
 \end{equation}

The claim now follows by taking the infima over all upper gradients of $u$ in \eqref{eq:I_1 is ok} and \eqref{eq:I_2 is ok}
and combining these two estimates.
\end{proof}

The next lemma can be seen as a counterpart of Proposition~\ref{prop:q-add}
for a Mazya-type condition (cf.\ \cite[\S 2.3]{mazya}):

\begin{proposition}\label{prop:mazja balls to E}
Let $\Omega\sub X$ be an open set and let
$\W_c(\Omega)=\{B_i\}_{i\in I}$ be a Whitney cover of  $\Omega$ 
with $c < \min\{(C_A)^{-1},(30\lambda)^{-1}\}$.
Assume the existence of a constant $C_0>0$ such that 
\begin{equation}\label{eq:assume hardy balls}
\int_U \dom(x)^{-p}\,d\mu(x)\leq C_0 \capp{p}(U,\Omega)
\end{equation}
whenever $U\sub\Omega$ is a union of Whitney balls.
Then there exists a constant $C>0$ such that
\begin{equation*}\label{eq:assume hardy**}
\int_E \dom(x)^{-p}\,d\mu(x)\leq C \capp{p}(E,\Omega)
\end{equation*}
whenever $E\subset\Omega$.
\end{proposition}

\begin{proof}
Let $E\subset\Omega$. If $\capp{p}(E,\Omega)=\infty$ the claim follows, and thus we may again assume that
$\capp{p}(E,\Omega)<\infty$. Let $u$ be the $p$-potential of $E$ with respect to $\Omega$,
and let $g_u$ be an upper gradient of $u$.
We denote $E_i=E\cap B_i$ and 
split the union $E=\bigcup_i E_i$ into two parts: 
We set $i\in I_1$ if
$u_{2B_i} < 1/2$ and $i\in I_2$ if $u_{2B_i} \geq 1/2$.

In the first case $i\in I_1$ we have 
$|u-u_{2B_i}|\geq 1/2$ in $E_i$, and so, using the $(p,p)$-Poincar\'e inequality
(a consequence of the $(1,p)$-Poincar\'e inequality by \cite[Theorem 5.1]{HaK})
and the bounded overlap of the balls $10\lambda B_i$,
we obtain 
\begin{equation}\label{eq:lotsa grad}
\begin{split}
\sum_{i\in I_1} \int_{E_i} \dom(x)^{-p} \,d\mu(x) & 
\leq C \sum_{i\in I_1} r_i^{-p} \int_{2B_i} |u-u_{2B_i}|^p \,d\mu\\ 
 & \leq C \sum_{i\in I_1} \int_{10\lambda B_i} g_u^p \,d\mu \leq C \int_{\Omega} g_u^p \,d\mu.
\end{split} 
\end{equation}

In the second case $i\in I_2$ it follows from
$u_{2B_i}\geq 1/2$ and $0\le u\le 1$ that for $U_i=\{u\ge 1/4\}\cap 2B_i$,
\[\begin{split}
\tfrac12 & \le\frac{1}{\mu(2B_i)}\left[ \int_{U_i}u\, d\mu+\int_{2B_i\setminus U_i}u\, d\mu\right] \\ 
             & \le\frac{1}{\mu(2B_i)}\left[\mu(U_i)+\tfrac14\mu(2B_i\setminus U_i)\right]
                 \le\frac{\mu(U_i)}{\mu(2B_i)}+\tfrac14,
\end{split} \]
from which we obtain
\[
\mu\big(\{x\in 2B_i : u(x)\geq 1/4\}\big)\geq \tfrac 1 4 \mu(2B_i).\]
Thus, by the weak Harnack inequality for superminimizers, we
obtain that
\[
\begin{split}
\inf_{B_i} u & \geq A^{-1}\bigg(\,\,\vint{2B_i} u^q\,d\mu \bigg)^{1/q} 
\geq C \bigg(\mu(2B_i)^{-1}\int_{\{u\geq 1/4\}\cap 2B_i} (1/4)^q\,d\mu \bigg)^{1/q}\geq C_1
\end{split}
\]
for each $i\in I_2$.
Hence the function $u/C_1$ is an admissible test function for
\[
\capp{p}\Big(\bigcup_{i\in I_2}B_i,\Omega\Big).  
\]
Using the bounded overlap of the balls $B_i$ 
and the assumption \eqref{eq:assume hardy balls},
we conclude that
\begin{equation}\begin{split}\label{eq:here too}
\sum_{i\in I_2} \int_{E_i} \dom(x)^{-p} \,d\mu(x) & \leq \sum_{i\in I_2} \int_{B_i} \dom(x)^{-p} \,d\mu(x)\\
& \leq C \int_{\bigcup_{i\in I_2}B_i} \dom(x)^{-p} \,d\mu(x)\\
& \leq C \capp{p}\Big(\bigcup_{i\in I_2}B_i,\Omega\Big)\leq \frac{C}{C_1^p}\ \int_\Omega g_u^{p}\, d\mu.
\end{split}\end{equation}
The lemma follows from estimates \eqref{eq:lotsa grad} and \eqref{eq:here too} by taking the infimum over all upper gradients of $u$.
\end{proof}

Combining the conditions from Propositions \ref{prop:q-add} and \ref{prop:mazja balls to E}, we
arrive at the main result of this section:

\begin{thm}\label{thm:all}
Let 
$\Omega\sub X$ be an open set, and let
$\W_c(\Omega)=\{B_i\}_{i\in \mathbb{N}}$ be a Whitney cover of $\Omega$ 
with $c < \min\{(C_A)^{-1},(30\lambda)^{-1}\}$.
Then the following conditions are (quantitatively)
equivalent:

(a) There exist $C>0$ such that 
\begin{equation*}\label{eq:assume hardy*}
\int_E \dom(x)^{-p}\,d\mu\leq C \capp{p}(E,\Omega)
\end{equation*}
for all $E\sub\Omega$.

(b) There exist $C>0$ such that
\begin{equation*}\label{eq:assume hardy balls*}
\int_U \dom(x)^{-p}\,d\mu\leq C \capp{p}(U,\Omega)
\end{equation*}
whenever $U=\bigcup_{i\in I} B_i$ for $B_i\in \W_c(\Omega)$ and $I\sub\N$.

(c) There exist $C>0$ such that
\[
\sum_{i\in \N} \capp{p}(E\cap B_i,\Omega) \leq C \capp{p}(E,\Omega)
\]
for all $E\sub\Omega$, and the capacity estimate \eqref{eq:est for cap} holds.

(d) There exist $C>0$ such that
\[
\sum_{i\in I} \capp{p}(B_i,\Omega) \leq C  \capp{p}(U,\Omega)
\]
whenever $U=\bigcup_{i\in I} B_i$ for $B_i\in \W_c(\Omega)$ and $I\sub\N$,
and the capacity estimate \eqref{eq:est for cap} holds.
\end{thm}

\begin{proof}
The implications \emph{(a)}$\implies$\emph{(b)} and \emph{(c)}$\implies$\emph{(d)}
are trivial and the implications converse to these are Proposition~\ref{prop:mazja balls to E} and
Proposition~\ref{prop:q-add}, respectively. As a link between these two equivalences we have
\emph{(b)}$\iff$\emph{(d)} from
Lemma~\ref{lemma:main for balls}, 
and here the 
lower bound of~\eqref{eq:est for cap} 
is needed to pass from \emph{(d)} to \emph{(b)}. Hence we assume the validity of~\eqref{eq:est for cap}
in parts (c) and (d). Recall that the validity of~\eqref{eq:est for cap} is guaranteed by Lemma~\ref{lemma:cap low bound} when the hypotheses of this lemma are satisfied, or by the uniform
$p$-fatness of $X\setminus\Omega$.
\end{proof}

\subsection{The Hardy connection}\label{sect:hardy}

We say that an open set $\Omega\sub X$ \emph{admits a $p$-Hardy inequality}
if there exists a constant $C>0$ such that the inequality
\begin{equation*}\label{eq:hardy*}
\int_{\Omega} |u(x)|^p\, \dom(x)^{-p}\,d\mu(x)
   \leq C\int_{\Omega} g_u^p\,d\mu,
\end{equation*}
holds for all $u\in N^{1,p}(\Omega)$ with $u=0$ on $X\setminus \Omega$
and for all upper gradiets $g_u$ of $u$.

Let us record the following Mazya-type characterization for Hardy inequalities.

\begin{lemma}\label{lemma:mazya and hardy}
An open set $\Omega\sub X$ admits a $p$-Hardy inequality
if and only if
\begin{equation}\label{eq:mazya and hardy}
 \int_E \dom(x)^{-p}\,d\mu(x) \leq C \capp{p}(E,\Omega)
\end{equation}  
for all $E\subset\Omega$.
\end{lemma}

\begin{proof}
For compact sets $K\sub\Omega$, the above characterization is proven in the 
metric space setting in~\cite[Theorem 4.1]{KoSh} (see also
\cite[\S 2.3]{mazya} in the Euclidean setting). 
Thus it suffices to show that if $\Omega$ admits a $p$-Hardy inequality and 
$E\sub\Omega$ is an arbitrary subset, then estimate \eqref{eq:mazya and hardy}
holds. If $\capp{p}(E,\Omega)=\infty$, then there is nothing to prove, and on the 
other hand if $\capp{p}(E,\Omega)<\infty$, then the $p$-Hardy inequality, used for capacity test functions
$u_k$ with $\lim_{k\to\infty} \int_\Omega g_{u_k}\,d\mu = \capp{p}(E,\Omega)$, 
yields the desired estimate \eqref{eq:mazya and hardy}.
\end{proof}

In other words, an open set $\Omega\sub X$ admits a $p$-Hardy inequality
if and only if the assertion \emph{(a)} of Theorem \ref{thm:all} holds. 
This leads immediately to the following corollary.

\begin{coro}\label{coro:q-add from hardy}
Let $\Omega\sub X$ be an open set and let 
$\W_c(\Omega)$ be a Whitney cover of $\Omega$
with a suitably small parameter $0<c\leq 1/2$.
Then $\Omega$ admits a $p$-Hardy inequality if and only if
the capacity $\capp{p}(\cdot,\Omega)$ is
quasiadditive with respect to $\W_c(\Omega)$
and the capacity estimate \eqref{eq:est for cap} holds for all balls $B\sub\Omega$.
\end{coro}

Since uniform $p$-fatness of the complement $X\setminus\Omega$ is a sufficient condition for $p$-Hardy inequalities in our setting (see \cite[Corollary 6.1]{BMS} and \cite{KLT}), we obtain a concrete sufficient condition for the quasiadditivity of the $p$-capacity:

\begin{coro}\label{coro:q-add from p-fat}
Let $\Omega\sub X$ be an open set and let 
$\W_c(\Omega)$ be a Whitney cover of $\Omega$
with a suitably small parameter $0<c\leq 1/2$.
Assume further that the complement $X\setminus\Omega$ is uniformly $p$-fat.
Then the capacity $\capp{p}(\cdot,\Omega)$ is
quasiadditive with respect to $\W_c(\Omega)$.
\end{coro}

\section{Quasiadditivity and the Aikawa dimension}\label{sect:q and a}

In this section we focus on open sets $\Om\subset X$ that satisfy 
$\codima(X\setminus\Om)>p$ (recall the definition from Section~\ref{sect:metric}). 
In this case we also have that  $\codimh(X\setminus\Om)>p$, and hence it follows 
from~\cite[Proposition~4.1]{lenb} that $\text{Cap}_p(X\setminus\Om)=0$. Therefore, as 
was remarked in Section~\ref{sect:nonlin-pot-thry}, we know that
$\capp{p}(E,\Om)=\capp{p}(E,X)$ for every $E\sub \Omega$. 
Recall from Section~\ref{sect:sobo} that if $X$ is $p$-parabolic, then actually $\capp{p}(E,X)=0$
for all bounded $E\sub X$. Thus, if $X$ is $p$-parabolic and $\Omega\sub X$ is such that 
$\text{Cap}_p(X\setminus\Om)=0$, then $\Om$ satisfies the quasiadditivity property trivially; 
the same is also true if $X$ is bounded and 
$\text{Cap}_p(X\setminus\Om)=0$. Hence in this section we assume that
$X$ is unbounded and $p$-hyperbolic.

We say that an open set $\Omega=X\setminus E$ and a related Whitney cover $\W=\W_c(\Omega)$
satisfy a \emph{discrete John condition} if there exist $L>1$, 
$a>1$, and $C>0$ such that for each $B\in\W$ we
find a chain $\mathcal{C}(B)=\{B_m\}_{m=0}^\infty$ of Whitney balls $B_m\in\W(\Omega)$, with $B_{0}=B$, such that 
$B_{m}\cap B_{m+1}\neq \emptyset$, $B\sub LB_m$, and $\rad(B_m)\geq Ca^{m} \rad(B)$ for each $m\in\mathbb{N}$.
This condition is satisfied, for instance, if $\Omega$ is an unbounded John domain (see \cite{V}); 
similar chain conditions have been used e.g.\ in \cite{HaK,HEI}.
Notice in particular that since our open sets are unbounded, there can not exist a `John center' as in the usual John condition for bounded domains; essentially the `point at infinity' acts as a John center. On the other hand, 
the domain $\Om=\{(x,y)\in\mathbb{R}^2\, :\, 0<y<|x|+1\}$ satisfies the discrete John condition, but is
not an unbounded John domain (in the sense of \cite{V}).

\begin{proposition}\label{prop:john and aikawa}
Let $\Omega\sub X$ be an open set with $\codima(X\setminus\Omega)>p$.  
Assume furthermore that $\Omega$ satisfies the above discrete John condition for a Whitney cover 
$\W_c(\Omega)$ with $c\leq (6\lambda)^{-1}$.
Then $\capp{p}(\cdot,\Omega)$ is quasiadditive with
respect to $\W_c(\Omega)$ and $\Omega$ admits a $p$-Hardy inequality.
\end{proposition}

\begin{proof}
By Theorem \ref{thm:all} and Corollary \ref{coro:q-add from hardy}, it suffices to show that
there is a constant $C>0$ such that 
if $U=\bigcup_{i\in I} B_i$, $B_i\in \W_c(\Omega)$, for some $I\subset\mathbb{N}$, 
then
\begin{equation}\label{eq:assume hardy balls again}
\int_U \dom(x)^{-p}\,d\mu(x)\leq C \capp{p}(U,\Omega). 
\end{equation}
Fix such a set $U$, and write $r_i=\rad(B_i)$, $i\in I$. We may clearly assume that
$\capp{p}(U,\Omega)<\infty$. Let $u$ be a capacity test-function for $U$.
Then $u_{B_i}=1$ for each $i\in I$
(and thus $u_{2B_i}\geq \alpha$ for some $\alpha>0$ since $0\leq u\leq 1$). 
On the other hand, since $u\in L^p(\Omega)$, we find, using the discrete John condition, 
for each $i\in I$ a chain of Whitney balls $B_{i,m}$, $m=0,1,\dots,M_i$,
where $B_{i,0}=B_i$, $B_{i,m-1}\cap B_{i,m}\neq\emptyset$,
$\rad(B_{i,m})\geq C a^m r_{i}$ for all $m=1,\dots,M_i$,
and $u_{2B_{i,M_i}}\leq \alpha/2$. Indeed, since $\rad(B_{i,m})\geq C a^m r_{i}$, we have by
\eqref{eq: d dim conv} that
\[
\frac{\mu(B_i)}{\mu(B_{i,m})}\leq C\bigg(\frac{r_i}{\rad(B_{i,m})}\bigg)^{Q_u} \leq C a^{-mQ_u},
\]
and thus, by H\"older's inequality,
\[
\vint{2B_{i,m}}u \leq C \mu(B_{i,m})^{-1/p}\|u\|_{L^p(X)} \xrightarrow{m\to\infty} 0,
\]
allowing us to choose the index $M_i$ as above.

By a standard chaining argument using
the $(1,p)$-Poincar\'e inequality (see e.g. \cite{HaK} or \cite{HEI}), we have that 
\begin{equation}\label{eq: chain}
 \alpha/2\leq |u_{2B_i}-u_{2B_{i,M_i}}| 
 \leq C \sum_{m=0}^{M_i} \rad(B_{i,m})\, \bigg(\;\vint{2\lambda B_{i,m}}g_u^p\,d\mu\bigg)^{1/p}.
\end{equation}
Comparing 
the sum on the right-hand side of \eqref{eq: chain} with the convergent 
geometric series $\sum_{m=0}^\infty a^{-m\delta}$,
we infer that if $\delta>0$, then there must exist a constant
$C_1>0$, independent of $u$ and $B_i$, and at least one index $m_i\in\N$
such that
\begin{equation}\label{eq: one big}
\rad(B_{i,m_i})\, \bigg(\;\vint{2\lambda B_{i,m_i}}g_u^p\,d\mu\bigg)^{1/p} \geq 
   C_1 a^{-m_i\delta} \geq C \left(\frac{r_i}{\rad(B_{i,m_i})}\right)^\delta.
\end{equation}
Let us now fix $q$ such that $\codima(X\setminus\Omega)>q>p$ and
set $\delta=(q-p)/p>0$. We thus obtain from 
\eqref{eq: one big} for each $B_i$ a ball $B_i^*=B_{i,m_i}$ 
with radius $r_i^*$ satisfying 
\begin{equation}\label{eq:bee star}
r_i^{q-p}
\leq C (r_i^*)^{q} \mu \big(B_{i}^*\big)^{-1} \int_{2\lambda B_{i}^*}g_u^p\,d\mu.
\end{equation}
Using estimate \eqref{eq:bee star}, 
and changing the summation to be over all Whitney balls, we calculate
\begin{equation}\label{eq:order}\begin{split}
\int_U \dom(x)^{-p}\,d\mu & \leq C \sum_{i} \mu(B_i)r_i^{-p}\\
& \leq C \sum_{i}\frac{\mu(B_i)r_i^{-q}}{\mu(B_i^*)(r_i^*)^{-q}} \int_{2\lambda B_{i}^*}g_u^p\,d\mu\\
& \leq C \sum_{B\in \W} \sum_{\{i:B=B_i^*\}}
 \frac{\mu(B_i)r_i^{-q}}{\mu(B)\rad(B)^{-q}} \int_{2\lambda B}g_u^p\,d\mu\\
& \leq C \sum_{B\in\W} \int_{2\lambda B}g_u^p\,d\mu  
        \sum_{\{i:B=B_i^*\}} \frac{\mu(B_i)r_i^{-q}}{\mu(B)\rad(B)^{-q}}. \\
\end{split}
\end{equation}

If $B=B_i^*$, that is, $B\in\W$ satisfies \eqref{eq:bee star} for the
ball $B_i$, then $B_i\sub LB$ by the chain condition.
Since $r_i^{-q}\approx \dom(x)$ for all 
$x\in B_i$, it follows from
the bounded overlap of the Whitney balls $B_i\sub LB$ 
and the assumption $\codima(X\setminus\Omega)>q$,
that
\begin{equation}\label{eq:aikawa helps}
\begin{split}
 \sum_{\{i:B=B_i^*\}} \frac{\mu(B_i)r_i^{-q}}{\mu(B)\rad(B)^{-q}}
 &\leq
  C \frac{1}{\mu(B) \rad(B)^{q}} \sum_{\{i:B=B_i^*\}} \int_{B_i} \dom(x)^{-q}\,d\mu\\
 &\leq
  C \frac{1}{\mu(B) \rad(B)^{q}}\int_{LB} \dom(x)^{-q}\,d\mu\\
 & \leq C \frac{1}{\mu(B) \rad(B)^{q}} \mu(LB)\rad(LB)^{-q}\leq C.
\end{split}
\end{equation} 

By the assumption $c\leq (6\lambda)^{-1}$ the overlap of the balls $2\lambda B$, where $B\in\W_c(\Omega)$, 
is uniformly bounded (Lemma \ref{lemma:no overlap}),
and so we conclude from \eqref{eq:order} and \eqref{eq:aikawa helps} that
\[
\int_U \dom(x)^{-p}\,d\mu \leq C \int_\Omega g_u^p\,d\mu.
\]
The claim \eqref{eq:assume hardy balls again} follows by taking the infimum over all capacity test functions for $U$ (and their upper gradients).
\end{proof}

 It has been shown in \cite[Section 6]{LT} (following the considerations of \cite{KZ}), that if $\Omega\sub X$ admits a $p$-Hardy inequality,
 then either $\codimh(X\setminus\Omega)<p-\delta$ or $\codima(X\setminus\Omega)>p+\delta$
 for some $\delta>0$ only depending on the data associated with the space $X$ and the Hardy inequality.  
 Moreover, there is also a local version of such a dichotomy for the dimension \cite[Theorem 6.2]{LT}. These
 results, together with the above Proposition~\ref{prop:john and aikawa} 
 (and see also the following Section~\ref{sect:combo}), 
 show clearly that the condition $\codima(X\setminus\Omega)>p$ is very natural in the context of Hardy inequalities and
 thus also for quasiadditivity. On the other hand, the case $\codimh(X\setminus\Omega)<p-\delta$ includes
 open sets with uniformly $p$-fat complements; cf.\ Corollary~\ref{coro:q-add from p-fat}.

 The main open question here is whether the discrete John condition is really necessary in Proposition~\ref{prop:john and aikawa}; we know of no counterexamples. In the Euclidean space $\R^n$ this extra condition is certainly not needed. Indeed, as commented at the end of \cite{KZ}, the dimension bound $\dima(\R^n\setminus\Omega)<n-p$ implies by \cite[Theorem 2]{A} that 
\[
\int_E \dom(x)^{-p}\,dx \leq C R_{1,p}(E) \leq C \capp{p}(E,\Omega)
\]
for all (measurable) $E\sub\Omega$; here $R_{1,p}(E)$ is a Riesz capacity of $E$ (cf.\ \cite{A} or \cite{AE} for the definition) and the second inequality is a well-known fact. Quasiadditivity for $\capp{p}(\cdot,\Omega)$ follows by Theorem \ref{thm:all}.

Nevertheless, Proposition~\ref{prop:john and aikawa} 
still gives a partial answer to the question of Koskela and Zhong \cite[Remark 2.8]{KZ}, i.e., a
$q$-Hardy inequality holds in their setting provided that $\Omega$ satisfies the discrete John 
condition (and the Minkowski dimension in \cite[Remark 2.8]{KZ} is replaced by the 
correct Aikawa (co)dimension).

\section{Combining fat and small parts of the complement}\label{sect:combo}

The results studied in Section~\ref{sect:qadd} gave us a criterion, uniform $p$-fatness of $X\setminus\Omega$,
under which $\Om$ supports quasiadditivity of $p$-capacity for the Whitney decompositions of $\Om$; this condition requires $X\setminus\Omega$ to be `large'. Conversely, in Section~\ref{sect:q and a} we gave a criterion, largeness of the Aikawa co-dimension of $X\setminus\Om$ (or, smallness of the Assouad dimension --- and hence `smallness' of $X\setminus\Om$), under which $\Om$ supports quasiadditivity for the Whitney decompositions of $\Om$. Nevertheless, requiring the whole complement to be either large or small rules out many interesting cases. For instance, sometimes the complement of a domain can be decomposed into two closed subsets
such that one of them is `large' and one is `small';
an easy example is the punctured ball $B(0,1)\setminus\{0\}\subset\mathbb{R}^n$. The aim of this final section is to explain how the results of the previous Sections~\ref{sect:qadd} and~\ref{sect:q and a} can be combined to address such more complicated sets. In the Euclidean case, some results into this direction can be found also in \cite{LMM}.
A full geometric characterization of domains supporting quasiadditivity of $p$-capacity for Whitney decompositions still seems to be beyond our reach.  However, in the next lemma we demonstrate a technique which applies to a broad class of sets. 

\begin{lemma}\label{lem:thick n thin}
Assume that $X$ is unbounded and that $1<p<Q_u$. Suppose that $\Om_0\sub X$ is an open set such that $X\setminus\Om_0$ is uniformly $p$-fat. Suppose also that $F\subset\ol\Om_0$ is a closed set with $\codima(F)>p$, and that $X\setminus F$ satisfies the discrete John condition of Section \ref{sect:q and a}. Then $\Om=\Om_0\setminus F$ 
satisfies a quasiadditivity property with respect to Whitney covers $\W_c(\Omega)$ with suitably small $c>0$.
\end{lemma}

\begin{proof}
Let $\W_c(\Omega)$ be a Whitney decomposition of $\Om$, where
$0<c<\min\{(C_A)^{-1},(30\lambda)^{-1}\}$. 
Set $\W^1$ to be the collection of all balls $B(x,r)\in \W$ satisfying
$\dist(x,X\setminus\Omega)=\dist(x,F)$ and, similarly, 
let $\W^2$ be the collection of all balls $B(x,r)\in \W$ satisfying
$\dist(x,X\setminus\Omega)=\dist(x,X\setminus\Omega_0)$.
It is clear that we can extend the collection $\W^1$ to a Whitney cover $\W^{1*}$ of $X\setminus F=:\Om_1$ and
the collection $\W^{2}$ to a Whitney cover $\W^{2*}$ of $\Om_0$, both with the same constant $c$ but possibly with larger overlap constants.

As before, to prove the quasiadditivity property, it suffices to consider unions of Whitney balls. Thus,
let $U=\bigcup_{i\in I} B_i$, where $B_i\in\W_c(\Omega)$ and $I\sub\N$; 
we may also assume that $\capp{p}(U,\Omega)<\infty$. Set 
$U_1=\bigcup_{B_i\in\W^{1}}B_i$, $U_2=\bigcup_{B_i\in\W^2}B_i$. 
Since
$\codima(X\setminus\Om_1)=\codima(F)>p$ and the discrete John condition holds for $\Omega_1$, we know, by Proposition~\ref{prop:john and aikawa}, that
\begin{equation}\label{eq:combo1}
   \sum_{\{i\in I:B_i\in\W^1\}}\capp{p}(B_i,\Om_1)\le C_1\capp{p}(U_1,\Om_1)\le C_1\capp{p}(U,\Om);
\end{equation}
here we use the facts $U_1\sub U$ and $\Om\subset\Om_1$.
On the other hand, an application of the results of Section~\ref{sect:qadd} yields
\begin{equation}\label{eq:combo2}
    \sum_{\{i\in I:B_i\in\W^2\}}\capp{p}(B_i,\Om_0)\le C_2\capp{p}(U_2,\Om_0)\le C_2\capp{p}(U,\Om).
\end{equation}
Here the last inequality follows since $U_2\sub U$ and $\Omega_0\setminus\Omega \subset F$ is of zero $p$-capacity. 
For the same reason we have in \eqref{eq:combo2} that $\capp{p}(B_i,\Om_0)=\capp{p}(B_i,\Om)$ for each $B_i\in \W^2$. To estimate the corresponding capacities on the left-hand side of \eqref{eq:combo1}, we use the capacity bounds from \eqref{eq:est for cap} (with respect to $\Omega$ and $\Omega_1$; note that the assumptions of Lemma \ref{lemma:cap low bound} are valid in the latter case) to obtain for all $B_i\in\W^1$ that
\[
\capp{p}(B_i,\Om)\leq C\rad(B_i)^{-p}\mu(B_i)\leq C_3 \capp{p}(B_i,\Om_1).
\]
In conclusion,
\[\begin{split}
 \sum_{i\in I}\capp{p}(B_i,\Omega) & \leq C_3 \bigg(\sum_{\{i\in I:B_i\in\W^1\}}\capp{p}(B_i,\Om_1) + \sum_{\{i\in I:B_i\in\W^2\}}\capp{p}(B_i,\Om_0)\bigg)\\ 
& \leq C_3(C_1+C_2)\capp{p}(U,\Om),
\end{split}\]
and the claim follows by Theorem \ref{thm:all}.
\end{proof}

\end{document}